\documentclass[12pt]{amsart}
\usepackage{amsmath}
\usepackage{amssymb}
\usepackage{latexsym}

\newcommand{\Zint}{\mathbb {Z}}    
     
\newcommand{\C}{\mathbb {C}}     
\newcommand{\Cplx}{\mathbb {C}}     
\newcommand{\halmos}{\rule{5pt}{5pt}}

\numberwithin{equation}{section}
\newtheorem{thm}{Theorem}[section]
\newtheorem{prop}[thm]{Proposition}

\newtheorem{conj}{\bf Conjecture}
\theoremstyle{definition} 

\theoremstyle{remark}

\begin{document}
\title[On zeros of polynomials associated with Heun class equations]
{On zeros of polynomials associated with Heun class equations}
\author{Mizuki Mori}
\author{Kouichi Takemura}
\address{Department of Mathematics, Ochanomizu University, 2-1-1 Otsuka, Bunkyo-ku, Tokyo 112-8610, Japan}
\email{takemura.kouichi@ocha.ac.jp}
\begin{abstract}
The Heun class equations have the accessory parameters, and we investigate the polynomials, whose variable is the accessory parameter, which appear as the coefficients of a local solution.
Our study is motivated by the result of Sch\"afke and Schmidt that the asymptotics of the coefficients of a local solution to some linear differential equation is related to global structures of solutions.
The zeros of the polynomials are related to the spectral problem of the Heun class equation.
We obtain exact results on the zeros from a perturbative approach, and we show some tables of the zeros by numerical calculations, which support our strategy.
\end{abstract}
\maketitle

\section{Introduction}

It is widely known that linear differential equations on the complex plane play important roles in theorical studies of the models in physics.
Among them, differential equations of the Heun class appear several times.
See \cite{SL,Tak5} and references therein.
They also appear in studies of cosmology including the Kerr black hole \cite{STU}.
Heun class equations include Heun's differential equation and its confluent ones.
Heun's differential equation is a standard form of the second order Fuchsian differential equation with four regular singularities $\{ 0,1,t,\infty \}$, and it is given as
\begin{equation}
\frac{d^2y}{dz^2} + \left( \frac{\gamma}{z}+\frac{\delta }{z-1}+\frac{\epsilon}{z-t}\right) \frac{dy}{dz} +\frac{\alpha \beta z -q}{z(z-1)(z-t)} y=0,
\label{eq:Heun}
\end{equation}
under the condition $\gamma +\delta +\epsilon =\alpha +\beta +1 $ (\cite{Ron,SL}).
The parameter $q$ is independent from the local exponents and it is called the accessory parameter.
It is well known that the Lam\'e equation is a special case of Heun's differential equation (see Section \ref{sec:Lame}).
The confluent Heun equation or the singly confluent Heun equation is given as
\begin{equation}
\frac{d^2y}{dz^2} + \left( 4p+ \frac{\gamma}{z}+\frac{\delta }{z-1} \right) \frac{dy}{dz} +\frac{4 p \alpha z - \sigma }{z(z-1)} y=0 . 
\end{equation}
See \cite[pp.~94]{Ron}.
On this equation, the singularities $z=0,1$ are regular, and the singularity $z= \infty $ is irregular.
The parameter $ \sigma$ is regarded as the accessory parameter.
The reduced singly confluent Heun equation is given as
\begin{equation}
\frac{d^2y}{dz^2} + \left( \frac{\gamma}{z}+\frac{\delta }{z-1} \right) \frac{dy}{dz} +\frac{-t z + \lambda }{z(z-1)} y=0 .
\label{eq:redsingconfHeun0}
\end{equation}
See \cite[pp.~106]{SL}.
The singularities $z=0,1$ are regular, and the singularity $z= \infty $ is irregular and ramified, although the irregular singularity $z= \infty $ of the singly confluent Heun equation is unramified.
The behaviour of the local solution about $z=\infty $ of the reduced singly confluent Heun equation is different from the behaviour of the local solution about $z=\infty $ of the singly confluent Heun equation.
The parameter $ \lambda $ is regarded as the accessory parameter.
It is well known that the Mathieu equation is a special case of the reduced singly confluent Heun equation (see Section \ref{sec:MWH}).
Other Heun class equations are known as doubly confluent, bi-confluent, tri-confluent and their reduced ones.
See \cite{Ron} and \cite{SL}.

Around 2010 and later, there were development in mathematical physics on the AGT correspondence \cite{AGT}.
As applications, new expression of solutions to Painlev\'e equations had been obtained \cite{ILT,LLNZ}, and new aspects of the study of Heun class of equations draw attentions \cite{PP}.
In particular, Bonelli, Iossa, Lichtig, Tanzini \cite{BIPT2} gave several formulas and conjectures on Heun class of equations.
Lisovyy and Naidiuk \cite{LN} studied further on this direction by using the method of Sch{\"a}fke and Schmidt \cite{SS}, which was published in 1980.

Motivated by Lisovyy-Naidiuk \cite{LN} and Sch{\"a}fke-Schmidt \cite{SS}, we study Heun equations from the perturbative viewpoint by focusing to the accessory parameter.
We now explain a result of Sch{\"a}fke and Schmidt briefly in the situation of Heun's differential equation.
It is known by the theory of the regular singularity that a basis of the local solutions of the Heun equation (\ref{eq:Heun}) about $z=0$ is obtained as
\begin{align}
& y^{\langle 0 \rangle} _1  = 1 + \sum _{k=1}^{\infty } c_k  z^k  , \quad  y^{\langle 0 \rangle} _2  = z^{1- \gamma } \Big( 1 +\sum _{k=1}^{\infty } c' _k z^k  \Big) , 
\end{align}
if $\gamma \not \in \Zint $.
Similarly, a basis of the local solutions of the Heun equation about $z=1$ is obtained as
\begin{align}
& y^{\langle 1 \rangle} _1  =  1+ \sum _{k=1}^{\infty } c '' _k (1-z)^k ,\quad  y^{\langle 1 \rangle} _2  = (1-z)^{1 -\delta } \Big( 1 +\sum _{k=1}^{\infty } c''' _k (1-z)^k  \Big) ,
\label{eq:Hlz=1}
\end{align}
if $\delta \not \in \Zint $.
We write a solution $ y^{\langle 0 \rangle} _1$ to the Heun equation by using a basis $\{  y^{\langle 1 \rangle} _1  , y^{\langle 1 \rangle} _2  \} $ as  
\begin{equation}
 y^{\langle 0 \rangle} _1 = d_{1}  y^{\langle 1 \rangle} _1  + d_2 y^{\langle 1 \rangle} _2 .
\end{equation}
The coefficients $d_1$ and $d_2 $ are complex numbers.
The situation such that $d_2 =0$ is special, because the holomorphic solution about $z=0$ is continued to be a holomorphic solution about $z=1$, and it would be special in the connection problem related to application to physics.
In this case, the solution to the Heun equation which is simultaneously holomorphic about $z=0$ and $1$ is called the Heun function.
By the way, Sch{\"a}fke and Schmidt \cite{SS} established the formula 
\begin{equation}
\lim _{k \to \infty} \frac{ k! }{ (\delta -1) \delta  \cdots (\delta  +k -2)} c_k = d_2 
\label{eq:limkckd2intro}
\end{equation}
for the Heun equation under the condition $|t| >1$ and $\gamma ,\delta \not \in \Zint $.
They established the formula as Theorem \ref{thm:SS} in a more general setting.
We can regard the values $d_2 $ and $c_k$ as the functions on the accessory parameter $q$.
By equation (\ref{eq:limkckd2intro}), it would be natural to expect that parts of the position of the zeros of $d_2$ on the parameter $q$ resemble to parts of the position of the zeros of $c_k$ for sufficiently large $k$.

In this paper, we replace the accessory parameter to the parameter $B$, and investigate the zeros of the coefficient $c_k (= c_k (B)) $ with respect to the parameter $B$.
It is shown that $c_k (B) $ is a polynomial of degree $k$ on the variable $B$.
We set $s=1/t$ and adopt a perturbative approach with respect to the variable $s$.
In our approach, we obtain the hypergeometric equation as $s \to 0$.
In the case $s=0$, the zeros of $c_k (B)$ are calculated explicitly, and we try to understand the zeros of $c_k (B)$ as a series expansion on the variable $s$.
On the singly confluent Heun equation and the reduced singly confluent Heun equation, we can adopt similar approachs, and we formulate them by the unified notation.

On the other hand, we can calculate the polynomials $c_k (B) $ explicitly for small $k$ by computer, and also calculate the zeros of $c_k (B)$ numerically.
Then, we can observe the behaviours of the zeros as $k$ varies.
We report them in Sections \ref{sec:Lame} and \ref{sec:MWH}.

This paper is organized as follows.
In Section \ref{sec:SS}, we review a remarkable result obtained in the paper \cite{SS} of Sch{\"a}fke and Schmidt.
In Section \ref{sec:TTR}, we discuss three-term relations for coefficients of the local solution for Heun class equations with paying attention to the accessory parameter.
In Section \ref{sec:pert}, we discuss perturbative expansions for zeros of $c_k (B) $ with respect to the accessory parameter.
The principal results are Proposition \ref{prop:Dkjm+1} and Theorems \ref{thm:B} and \ref{thm:expa2}.
In Section \ref{sec:Lame}, we discuss the Lam\'e equation and numerical calculations on it.
In Section \ref{sec:MWH}, we discuss the Mathieu equation, the Whittaker-Hill equation and numerical calculations on them.
In Section \ref{sec:rmk}, we give concluding remarks.

\section{Remarkable result of Sch{\"a}fke and Schmidt} \label{sec:SS}

In \cite{SS}, R.~Sch{\"a}fke and D.~Schmidt investigated the differential equation
\begin{equation}
 y''(z) + \Big( \frac{1-\mu _0}{z} +\frac{1-\mu _1}{z-1} + a(z) \Big) y'(z) + \frac{b(z)}{z(z-1)} y(z) = 0,
\label{eq:SS}
\end{equation}
where $a(z)$ and $b(z)$ are holomorphic in a domain which contains the circle $\{ z \in \C \: | \: |z|<R \} $ such that $R>1$, under the condition that $\mu_0 $ and $\mu _1$ are not integers.
Then, the singularities $z=0$ and $1$ are regular with the exponents $\{0 , \mu _0 \} $ and $\{0 , \mu _1 \} $.

We review a remarkable result of Sch{\"a}fke and Schmidt in a different notation.
Since the exponents about $z=0$ are $\{0 , \mu _0 \} $ and $\mu _0 \not \in \Zint $, there exist local solutions about $z=0$ written as 
\begin{align}
& y^{\langle 0 \rangle} _1  = 1 + \sum _{k=1}^{\infty } c_k  z^k  ,  \quad  y^{\langle 0 \rangle} _2  = z^{\mu _0} \Big( 1 +\sum _{k=1}^{\infty } c' _k z^k  \Big) . \label{eq:locexpck} 
\end{align}
Since the exponents about $z=1$ are $\{0 , \mu _1 \} $ and $\mu _1 \not \in \Zint $, there exist local solutions about $z=1$ written as 
\begin{align}
& y^{\langle 1 \rangle} _1  =  1+ \sum _{k=1}^{\infty } c '' _k (1-z)^k ,  \quad y^{\langle 1 \rangle} _2  = (1-z)^{\mu _1} \Big( 1 +\sum _{k=1}^{\infty } c''' _k (1-z)^k  \Big) . 
\end{align}
By the assumption of the differential equation (\ref{eq:SS}), the function $ y^{\langle 0 \rangle} _1 $ is holomorphic on $\{ z \in \C \: | \: |z|<1 \} $.
We write a solution $ y^{\langle 0 \rangle} _1$ to equation (\ref{eq:SS}) by using a basis $\{  y^{\langle 1 \rangle} _1  , y^{\langle 1 \rangle} _2  \} $ as  
\begin{equation}
 y^{\langle 0 \rangle} _1 = d_{1}  y^{\langle 1 \rangle} _1  + d_2 y^{\langle 1 \rangle} _2 .
\label{eq:y0d1d2}
\end{equation}
Sch{\"a}fke and Schmidt established that the coefficient $d_2$ is expressible in terms of some limit of the coefficients of the local expansion in equation (\ref{eq:locexpck}).
\begin{thm} \label{thm:SS} (c.f.~\cite[Theorem 2.15]{SS})
\begin{equation}
\lim _{k \to \infty} \frac{ k! }{ (-\mu_1) (-\mu_1 +1) \dots (-\mu _1 +k -1)} c_k = d_2 .
\label{eq:limkckd2}
\end{equation}
\end{thm}
The expression of the theorem in \cite[Theorem 2.15]{SS} is different from the theorem above.
In the rest of this section, we rewrite the theorem in \cite[Theorem 2.15]{SS} to the form of Theorem \ref{thm:SS}.
Write local holomorphic solutions about $z=0$ and $z=1$ as
\begin{equation}
y_{01} (z) = \sum _{k=0}^{\infty } \frac{\tau ^0 _k }{\Gamma (1- \mu _0+k)} z^k , \quad  y_{11} (z) = \sum _{k=0}^{\infty } \frac{\tau ^1 _k }{\Gamma (1- \mu _1+k)} (1-z)^k ,
\end{equation}
where $\tau ^0 _0 = 1 = \tau ^1 _0$ and $\Gamma ( \: \cdot \:  ) $ is the gamma function.
Since the exponents about $z=1$ are $0 $ and $ \mu _1 $, we can take a basis of solutions to equation (\ref{eq:SS}) as $y_{11}(z) $ and $y_{12} (z) $ where
\begin{align}
& y_{12} (z) = (1-z)^{\mu _1} \Big( \frac{1}{\Gamma (1+\mu _1)} +\sum _{k=1}^{\infty }  \frac{\tilde{\tau }^1 _k  }{\Gamma (1+\mu _1 +k )} (1-z)^k  \Big) . \nonumber 
\end{align}
Let
\begin{equation}
 w[ y_1, y_2 ] (z) = y_1 (z) y_2 '(z) - y_1 '(z) y_2 (z) 
\end{equation}
be the Wronskian of two arbitrary solutions $y_1 (z)$, $y_2 (z)$ of equation (\ref{eq:SS}).
Then, the Wronskian satisfies $ w' +\{ (1-\mu _0)/z + (1-\mu _1)/(z-1) + a(z) \} w =0 $, and there exists a unique constant $[ y_1 (z), y_2 (z) ]  \in \C $ such that
\begin{equation}
 w[ y_1, y_2 ] (z) = [ y_1 (z) , y_2 (z) ] z^{\mu _0 -1} (1-z)^{\mu _1 -1} \exp \Big( - \int _1 ^z a(\zeta ) d\zeta \Big) .
\end{equation}
Set
\begin{equation}
Q=  [ y_{01} (z) , y_{11} (z) ] .
\end{equation}
The elements of the connection matrix on Floquet solutions are expressed by using $Q$ (\cite[Proposition 2.14]{SS}).
We write a solution $y_{01} (z) $ to equation (\ref{eq:SS}) by using a basis  $\{ y_{11} (z) , y_{12} (z) \} $ as  
\begin{equation}
 y_{01} (z) = \tilde{d}_{1} y_{11} (z) +\tilde{d}_2 y_{12} (z) .
\label{eq:ty0td1td2}
\end{equation}
It was shown in \cite[Proposition 2.14]{SS} that 
\begin{equation}
\tilde{d}_2 =  - Q \frac{\pi }{\sin (\pi \mu _1)} .
\label{eq:td2}
\end{equation}
Sch{\"a}fke and Schmidt established that the constant $Q$ is expressed by using the asymptotics of the coefficient $\tau _k^0$ in the expansion of the holomorphic solution $y_{01} $ about the origin. 
Namely, they established in \cite[Theorem 2.15]{SS} that 
\begin{equation}
\lim _{k \to \infty } \frac{\Gamma (k+1)}{\Gamma (k+1 - \mu _0)\Gamma (k - \mu _1)} \tau _k^0 = Q .
\label{eq:SSTh215}
\end{equation} 
On the solutions, we have
\begin{align}
& y^{\langle 0 \rangle} _1 = \Gamma (1-\mu _0 )y_{01} (z) , \; y^{\langle 1 \rangle} _1 = \Gamma (1 -\mu _1 )y_{11} (z) , \; y^{\langle 1 \rangle} _2 = \Gamma (1+\mu _1 )y_{12} (z) , \label{eq:rely0y01etc} \\
& c_k = \Gamma (1-\mu _0 ) \tau ^0 _k , \: (k=0,1,2,\dots ) .  \nonumber
\end{align}
Then, it follows from equations (\ref{eq:y0d1d2}), (\ref{eq:ty0td1td2}) and (\ref{eq:td2}) that 
\begin{equation}
d_2 =  Q \frac{\pi  \Gamma (1-\mu _0 )}{\sin (-\pi \mu _1 ) \Gamma (1+\mu _1 )}  =  Q  \Gamma (- \mu _1 ) \Gamma (1-\mu _0 ) .
\end{equation}
Hence, it follows from equations (\ref{eq:SSTh215}) and (\ref{eq:rely0y01etc}) that
\begin{equation}
\lim _{k \to \infty} \frac{ k!\Gamma (- \mu _1 )  }{ \Gamma (-\mu _1 +k )} c_k = d_2 .
\end{equation}
Namely, we obtain equation (\ref{eq:limkckd2}).

\section{Three-term relations for Heun class equations} \label{sec:TTR}

Recall that Heun's differential equation was given by
\begin{equation}
\frac{d^2y}{dz^2} + \left( \frac{\gamma}{z}+\frac{\delta }{z-1}+\frac{\epsilon}{z-t}\right) \frac{dy}{dz} +\frac{\alpha \beta z -q}{z(z-1)(z-t)} y=0,
\label{Heun}
\end{equation}
with the relation $\gamma +\delta +\epsilon =\alpha +\beta +1$.
We look into the solution written as 
\begin{align}
& y = \sum_{i=0}^{\infty } c_i z^i , \; c_0=1. 
\label{eq:z0}
\end{align}
By substituting it into equation (\ref{Heun}) multiplied by $z(z-1)(z-t)$, we obtain the recursive relations
\begin{align}
& -q c_0 +t\gamma c_1=0 , \label{eq:Hlci} \\
& (i-1+\alpha )(i-1+\beta )c_{i-1} -[i\{ (i-1+\gamma )(1+t) +t\delta +\epsilon \} +q]c_i \nonumber \\
& \qquad \qquad \qquad \qquad \qquad \qquad +(i+1)(i+\gamma )t c_{i+1}=0 , \; i=1,2,3, \dots . \nonumber
\end{align}
The solution $y$ in equation (\ref{eq:z0}) was denoted by $Hl(t,q;\alpha ,\beta ,\gamma ,\delta ;z)$ in \cite{Ron}, and it was used to express other local solutions.
The coefficients $c_i$ are polynomial in the variable $q$, and the degree is equal to $i$.
Set $q=Bt$, $t=1/s$ and 
\begin{equation}
 y = \sum _{k=0}^{\infty } c_k (B)  z^k , \quad c_0(B)=1 .
\label{eq:yckHeun}
\end{equation}
Then, the Heun equation is written as
\begin{equation}
\frac{d^2y}{dz^2} + \left( \frac{\gamma}{z}+\frac{\delta }{z-1} - \frac{s\epsilon}{1-sz}\right) \frac{dy}{dz} +\frac{B - s\alpha \beta z}{z(z-1)(1-sz)} y=0.
\label{HeunBs}
\end{equation}
The recursive relation is written as
\begin{align}
& c_{m+1} (B) = \frac{(B+D_m +s E_m) c_{m} (B) -s F_m c_{m-1} (B)}{(m+1)(m+\gamma )} , \; m=0,1,2, \dots ,
\end{align}
where $ c_{-1} (B)=0$ and
\begin{align}
& D_{m}= m(m-1+\gamma+\delta),\; E_{m}= m(m-1+\gamma+\epsilon), \; F_{m}= (m-1+\alpha)(m-1+\beta).
\end{align}
The parameter $B$ plays the role of the accessory parameter.
Let $y^{\langle 1 \rangle} _1$ and $y^{\langle 1 \rangle} _2 $ be a basis of the local solutions about $z=1$ which are written as equation (\ref{eq:Hlz=1}).
Then, the solution $y$ in equation (\ref{eq:yckHeun}) is written as
\begin{equation}
 y = d_{1}(B)  y^{\langle 1 \rangle} _1  + d_2 (B)  y^{\langle 1 \rangle} _2 ,
\label{eq:d1Bd2B}
\end{equation}
and it follows from Theorem \ref{thm:SS} that 
\begin{equation}
\lim _{k \to \infty} \frac{ k! }{ (\delta -1) \delta  \cdots (\delta  +k -2)} c_k (B) = d_2 (B)
\label{eq:limkckd2B}
\end{equation}
for each $B \in \Cplx $, if $|s|<1$ and $\gamma , \delta \not \in \Zint $.
We may expect that each position of the zeros of $d_2 (B) $ is approximated by the position of the zeros of $c_k (B) $ as $k \to \infty $.

If $s=0$, then equation (\ref{HeunBs}) is in the form of the hypergeometric equation of Gauss, and the parameter $B$ is related with the local exponents about $z=\infty $.
In this case, the values $d_1 (B)$ and $d_2 (B)$ can be calculated explicitly by the connection formula.

Recall that the singly confluent Heun equation was given by
\begin{equation}
\frac{d^2y}{dz^2} + \left( 4p+ \frac{\gamma}{z}+\frac{\delta }{z-1} \right) \frac{dy}{dz} +\frac{4 p \alpha z - \sigma }{z(z-1)} y=0.
\label{eq:singconfHeun}
\end{equation}
The solution written as equation (\ref{eq:z0}) was denoted by $Hc(p, \alpha,\gamma ,\delta ,\sigma  ;z)$ in \cite{Ron}, and it was used to express other local solutions.
We now change the parameters by $4p= -s $ and $\sigma = -B$.
Then, the singly confluent Heun equation is written as
\begin{equation}
\frac{d^2y}{dz^2} + \left( -s + \frac{\gamma}{z}+\frac{\delta }{z-1} \right) \frac{dy}{dz} +\frac{-s \alpha z + B }{z(z-1)} y=0 .
\end{equation}
Set 
\begin{equation}
 y = \sum _{k=0}^{\infty } c_k (B)  z^k .
\label{eq:yckCH}
\end{equation}
The recursive relation is written as
\begin{align}
& c_{m+1} (B) = \frac{(B+D_m +s E_m) c_{m} (B) -s F_m c_{m-1} (B)}{(m+1)(m+\gamma )}  , \; m=0,1,2, \dots ,
\label{eq:cm+1CH}
\end{align}
where $ c_{-1} (B)=0$ and
\begin{align}
& D_{m}= m(m-1+\gamma+\delta),\; E_{m}= m , \; F_{m}= m-1+\alpha \label{eq:CHDEF} . 
\end{align}
Let $y^{\langle 1 \rangle} _1$ and $y^{\langle 1 \rangle} _2 $ be a basis of the local solutions about $z=1$ which are written as equation (\ref{eq:Hlz=1}).
Then, the solution $y$ in equation (\ref{eq:yckCH}) is written as equation (\ref{eq:d1Bd2B}) and it follows from Theorem \ref{thm:SS} that equation (\ref{eq:limkckd2B}) holds for each $B, s  \in \Cplx $, if $\gamma , \delta \not \in \Zint $.

The reduced singly confluent Heun equation was given in equation (\ref{eq:redsingconfHeun0}), and we change the parameters of it by $t=s$ and $\lambda =B$.
Then, we obtain the differential equation
\begin{equation}
\frac{d^2y}{dz^2} + \left( \frac{\gamma}{z}+\frac{\delta }{z-1} \right) \frac{dy}{dz} +\frac{-s z + B }{z(z-1)} y=0 .
\end{equation}
Write a local solution $y$ as equation (\ref{eq:yckCH}) and substitute it into the reduced singly confluent Heun equation.
Then, we have the recursive relation
\begin{align}
& c_{m+1} (B) = \frac{\{ B+m(m-1+\gamma+\delta) \} c_{m} (B) -s c_{m-1} (B)}{(m+1)(m+\gamma )},  \; m=0,1,2, \dots ,
\label{eq:cm+1rCH}
\end{align}
where $c_{-1} (B)=0$.
Note that this equation is also written as equation (\ref{eq:cm+1CH}) where 
\begin{align}
& D_{m}= m(m-1+\gamma+\delta), \; E_m=0 , \; F_{m}=1 . 
\end{align}
On this case, we also have the formula corresponding to equation (\ref{eq:limkckd2B}) for each $B, s  \in \Cplx $, if $\gamma , \delta \not \in \Zint $.

\section{Perturbative expansions for zeros} \label{sec:pert}

We investigate the zeros of the polynomials $c_{m+1} (B) $ which are determined by $c_{-1} (B)=0$, $c_{0} (B)=1$ and
\begin{align}
& c_{m+1} (B) = \frac{(B+D_m +s E_m) c_{m} (B) -s F_m c_{m-1} (B)}{(m+1)(m+\gamma )} ,  \; m=0,1,2, \dots .
\label{eq:cm+1mm-1}
\end{align}
If $s=0$, then we have 
\begin{align}
& c_{m+1} (B) | _{s=0} = R_{m+1} (B+D_0) (B+D_1 ) \cdots (B+D_{m}), \; R_{m+1} = \frac{1}{(m +1) ! (\gamma )_{m+1} } , 
\end{align}
where $(a)_{m+1}  = a (a+1) \cdots (a+m)$.
Therefore, the zeros of $ c_{m+1} (B) | _{s=0} $ are $-D_0, -D_1,$ $ \dots , -D_m$, and we try to investigate the zeros of $c_{m+1} (B)  $ as power series of the variable $s$.

The function $c_{m+1} (B) $ is a polynomial in $B$ and $s$.
The degree of the polynomial $c_{m+1} (B)$ in $B$ is $m+1$, and the coefficient of $B^{m+1}$ is equal to $R_{m+1}$.
The polynomial $c_{m+1} (B)$ is expressed as 
\begin{align}
& c_{m+1} (B) =  R_{m+1} (B+D_0) (B+D_1 ) \cdots (B+D_{m}) + d_{m+1} (B,s) s , \label{eq:cm012} 
\end{align}
where $d_{m+1} (B,s)$ is a polynomial in $B$ and $s$.
The following proposition is shown immediately.
\begin{prop}\label{prop1}
Let $p$ be a non-negative integer.
If $k\in \{  0, 1, \dots ,p \}$, then we have
\begin{align}
& c_{p+1}(-D_{k} +As)= O(s)
\end{align}
for any constant $A$.
\end{prop}
We investigate the zeros of the polynomials $c_{m+1} (B) $ under the assumption that $D_0, D_1, D_2, \dots $ are mutually different, which is equivalent to $\gamma + \delta \not \in \Zint _{\leq 0} $ for $D_{m}= m(m-1+\gamma+\delta) $.
In the propositions and the theorems below in this section, we also assume that $D_0, D_1, D_2, \dots $ are mutually different.
Let $k \in \{ 0,1, \dots ,m \}$.
We have $c_{m+1} (-D_k) | _{s=0} =0 $ and $\frac{\partial }{\partial B} c_{m+1} (B) | _{B=-D_k, s=0} = R_{m+1} (D_0 -D_k) \cdots (D_{k-1} -D_k) (D_{k+1} -D_k) \cdots (D_{m} -D_k) \neq 0$.
Hence, it follows from the analytic implicit function theorem that there exists a positive number $\varepsilon $ and an analytic function $D_{k,m+1} (s) $ defined on the open disk $|s|< \varepsilon $ such that $D_{k,m+1} (0)  = D_k  $ and
\begin{align}
& c_{m+1} (-D_{k,m+1} (s) ) =0. 
\label{eq:cmDm}
\end{align}
Write
\begin{align}
& D_{k,m+1} (s) = D_k + \sum _{j=1}^{\infty } D_k^{[j],(m+1)} s^j .
\end{align}
We discuss a characterization of the coefficients $D_k^{[j],(m+1)} $.
We replace $D_k^{[j],(m+1)}$ to the indeterminant $x^{[j]}$ and we write
\begin{align}
& c_{m+1} \Bigl( - D_k - \sum _{j'=1}^{\infty } x^{[j']} s^{j'}  \Bigr) = \sum _{j=0}^{\infty} p_j s^j .
\end{align}
Then, we have $p_0=0$ and $p_j$ $(j \geq 1)$ is a polynomial of $x^{[1]}, \dots , x^{[j]}$ such that the term containing $x^{[j]} $ is written as 
\begin{align}
& - R_{m+1} (D_0 -D_k) \cdots (D_{k-1} -D_k) (D_{k+1} -D_k) \cdots (D_{m} -D_k) x^{[j]} ,
\end{align}
which follows from equation (\ref{eq:cm012}).
Therefore, if we impose the condition $p_1 =0, p_2 =0, \dots $, then the indeterminants $x^{[1]}, x^{[2]}, \dots $ are determined recursively, and it follows from equation (\ref{eq:cmDm}) that $ x^{[j]} = D_k^{[j],(m+1)}$ for $j=1,2,\dots $.
Thus, we obtain the following proposition.
\begin{prop} \label{prop:cmOj}
Let $j$ be a positive integer and $u^{[1]}, u^{[2]}, \dots $ be a sequence of complex numbers.
Then, the condition
\begin{align}
& c_{m+1} \Bigl( -D_k - \sum _{j'=1}^{\infty } u^{[j']} s^{j'}  \Bigr) = O(s^{j+1}) 
\end{align}
is equivalent to the condition
$ u^{[j']} = D_k^{[j'],(m+1)}$ for $j'=1,\dots ,j$.
\end{prop}
We calculate the value $D_k^{[1],(m+1) }$ for $k=0,\dots ,m$.
Firstly, we calculate $D_{m-1}^{[1],(m+1)} $.
\begin{prop} $($\cite{M}$)$ \label{prop:B1}
Set
\begin{align}
& B_{m-1}^{(m+1)}= -D_{m-1}+\Bigl\{-E_{m-1}-\frac{(m-1)(m-2+\gamma)}{D_{m-1}-D_{m-2}}F_{m-1}-\frac{m(m-1+\gamma)}{D_{m-1} -D_{m}}F_{m}\Bigr\}s ,
\label{B_m-1}
\end{align}
for $m \in \Zint _{\geq 2} $ and 
\begin{align}
& B_{0}^{(2)}= -D_{0}+\Bigl\{-E_{0}-\frac{\gamma }{D_0 -D_1}F_{1}\Bigr\}s .
\label{B_m-11}
\end{align}
If $m \in \Zint _{\geq 1}$, then we have
\begin{align}
  c_{m+1}(B_{m-1}^{(m+1)})= O(s^2)  . \label{eq:propOs^2}
\end{align}
\end{prop}
\begin{proof}
Set $B_{m-1}^{(m+1)}= -D_{m-1} + s A$ and determine the condition for $A$ such that $c_{m+1}(-D_{m-1} + s A )= O(s^2)$.
It follows from equation (\ref{eq:cm+1mm-1}) that
  \begin{align}
    & c_{m+1}(-D_{m-1}+sA) \label{C_{m+1}(-D_{m-1}+sA)} \\
    & =\frac{(-D_{m-1}+sA+D_{m}+sE_{m})c_{m}(-D_{m-1}+sA) - sF_{m}c_{m-1}(-D_{m-1}+sA)}{(m+1)(m+\gamma)} . \nonumber
  \end{align}
Assume that $m \in \Zint _{\geq 2} $.
By substituting $B= -D_{m-1}+sA$ in the recursive equation (\ref{eq:cm+1mm-1}) where $m$ is replaced with $m-1$ or $m-2$, we have
  \begin{align}
    & c_{m}(-D_{m-1}+sA) \label{C_{m}(-D_{m-1}+sA)} \\
    & = s\frac{A+E_{m-1}}{m(m+\gamma +1)}c_{m-1}(-D_{m-1}+sA)-s\frac{F_{m-1}}{m(m+\gamma+1)}c_{m-2}(-D_{m-1}+sA) , \nonumber \\
    & c_{m-1}(-D_{m-1}+sA) \label{C_{m-1}(-D_{m-1}+sA)} \\
    & = \frac{\{ D_{m-2} -D_{m-1}+s(E_{m-2}+A)\}c_{m-2}(-D_{m-1}+sA) + O(s)}{(m-1)(m+\gamma -2)} . \nonumber
  \end{align}
It follows from equation (\ref{C_{m-1}(-D_{m-1}+sA)}) that
  \begin{align}
    c_{m-2}(-D_{m-1}+sA)= \frac{(m-1)(m-2+\gamma)}{D_{m-2}-D_{m-1}}c_{m-1}(-D_{m-1}+sA)+O(s) \label{C_{m-2}(-D_{m-1}+sA)}
  \end{align}
We substitute equations (\ref{C_{m}(-D_{m-1}+sA)}) and (\ref{C_{m-2}(-D_{m-1}+sA)}) into equation (\ref{C_{m+1}(-D_{m-1}+sA)}).
We obtain
  \begin{align}
    &c_{m+1}(-D_{m-1}+sA) = \frac{c_{m-1}(-D_{m-1}+sA)}{(m+1)(m+\gamma)}\Big[ \{ D_{m} -D_{m-1}  +s(A+E_{m})\} \label{cm+1cm} \\ 
    &\cdot \Big\{ s\frac{A+E_{m-1}}{m(m+\gamma -1)} -s\frac{F_{m-1}}{m(m+\gamma -1)}\frac{(m-1)(m-2+\gamma)}{D_{m-2} -D_{m-1}} \Big\} -sF_{m} \Big] +O(s^2) . \nonumber 
  \end{align}
If 
  \begin{align}
&  (D_{m}-D_{m-1} )\Big\{  \frac{A+E_{m-1}}{m(m+\gamma -1)} -\frac{F_{m-1}}{m(m+\gamma -1)} \frac{(m-1)(m-2+\gamma)}{D_{m-2} -D_{m-1}} \Big\} -F_{m}= 0,
  \end{align}
then the right hand side of equation (\ref{cm+1cm}) is of order $O(s^2)$, and it is equivalent to
  \begin{align}
& A= -E_{m-1} -F_{m-1}\frac{(m-1)(m-2+\gamma)}{D_{m-1} -D_{m-2}} -F_{m}\frac{m(m +\gamma -1)}{D_{m-1} -D_{m}}.  \label{eq:A}
  \end{align}
Therefore, we obtain the proposition for the case $m \in \Zint _{\geq 2} $.
The case $m=1$ is shown directly.
\end{proof}
\begin{prop} $($\cite{M}$)$ \label{prop:B2}
Let $m \in \Zint _{\geq 1} $.
Set
\begin{align}
& B_{m}^{(m+1)}= -D_{m}+\Bigl\{-E_{m}-\frac{m(m-1+\gamma)}{D_{m}-D_{m- 1}}F_{m}\Bigr\}s .
\end{align}
Then, we have
\begin{align}
  c_{m+1}(B_{m}^{(m+1)})= O(s^2) . 
\end{align}
\end{prop}
\begin{proof}
It follows that
  \begin{align}
    &c_{m+1}(-D_{m} +As) = \frac{ s(A +E_{m})c_{m}(-D_{m} +As) -sF_{m} c_{m-1}(-D_{m} +As) }{(m+1)(m+\gamma)} , \\
    &c_{m}(-D_{m} +As) = \frac{\{ -D_{m} +D_{m-1} +s(A+E_{m})\}c_{m-1}(-D_{m} +As) +O(s)}{m(m +\gamma -1)} . \nonumber
  \end{align}
Therefore, if
\begin{align}
&  A= -E_{m} -\frac{m(m +\gamma -1)}{D_{m} -D_{m-1}}F_{m} ,
\end{align}  
then we have $c_{m+1}(B_{m}^{(m+1)})= O(s^2)$.
\end{proof}
\begin{prop} $($\cite{M}$)$ \label{prop:Dk1m+1}
If $k \leq m-1$ and $m \leq m' $, then we have $D_k^{[1],(m+1)} = D_k^{[1],(m'+1)}  $.
\end{prop}
\begin{proof}
It is enough to show that $D_k^{[1],(m+1)} = D_k^{[1],(m+2)}  $ for any integer $m$ such that $m  \geq k + 1  $.
We substitute $B= -D_k - D_k^{[1],(m+1)} s $ in the equation
\begin{align}
& c_{m+2} (B) = \frac{1}{(m+2)(m+1+\gamma )} \{ (B+D_{m+1} +s E_{m+1}) c_{m+1} (B) -s F_{m+1} c_{m} (B) \} .
\label{eq:cm+210}
\end{align}
By Proposition \ref{prop:cmOj}, we have $ c_{m+1} ( -D_k - D_k^{[1],(m+1)} s ) = O(s^2) $.
It follows from Proposition \ref{prop1} and the condition $k \leq m-1 $ that $c_m (-D_k -A s ) = O(s)$ for any $A$, and we have $s F_{m+1} c_{m} ( -D_k - D_k^{[1],(m+1)} s) = O(s^2)$.
Hence, we obtain $ c_{m+2} ( -D_k - D_k^{[1],(m+1)} s) = O(s^2) $ by equation (\ref{eq:cm+210}).
Therefore, we have $D_k^{[1],(m+1)} = D_k^{[1],(m+2)} $ by applying Proposition \ref{prop:cmOj}.
\end{proof}
It follows from Proposition \ref{prop:Dk1m+1} that $D_k^{[1],(m+1)} = D_k^{[1],(k+2)}  $, if  $m \geq k+1$.
By combining with Propositions \ref{prop:cmOj}, \ref{prop:B1} and \ref{prop:B2}, we obtain the following theorem.
\begin{thm} $($\cite{M}$)$ \label{thm:B}
On the solution to $c_{m+1}(B)= 0 $ expressed as
\begin{align}
& B = -D_{k} - D_k^{[1],(m+1)} s +O(s^2), \quad k=0,1,\dots , m,
\end{align}
we have
\begin{align}
& D_k^{[1],(m+1)} = D_k^{[1],(k+2)} = E_{k}+\frac{k(k-1+\gamma)}{D_{k}-D_{k-1}}F_{k}+\frac{(k+1)(k+\gamma)}{D_{k} -D_{k+1}}F_{k+1} 
\end{align}
for $k=1,\dots ,m-1$ and 
\begin{align}
& D_0^{[1],(m+1)} = E_{0}+\frac{\gamma }{D_{0} -D_{1}}F_{1} , \quad  D_m^{[1],(m+1)} = E_{m}+\frac{m(m-1+\gamma)}{D_{m}-D_{m-1}}F_{m}.
\end{align}

\end{thm}
Proposition \ref{prop:Dk1m+1} is extended to the coefficients of higher degree.
\begin{prop} \label{prop:Dkjm+1}
Let $j $ be a positive integer.
If $k + j \leq m  \leq m' $, then we have $D_k^{[j],(m+1)} = D_k^{[j],(m'+1)}  $.
In particular, we have $D_k^{[j],(m+1)} = D_k^{[j],(k+j+1)}$ under the condition $m \geq k + j $.
\end{prop}
\begin{proof}
It is enough to show that  $D_k^{[j],(m+1)} = D_k^{[j],(m+2)}  $ for any integer $m$ such that $m  \geq k + j  $ by the induction on $j$.
The case $j=1$ was shown in Proposition \ref{prop:Dk1m+1}.
Assume that the equality holds true for $j \leq \ell -1$.
Assume that $m  \geq k + \ell $.
It follows from the assumption of the induction that $D_k^{[j],(m)} = D_k^{[j],(m+1)}  $ for  $j \leq \ell -1$.
By the definition of $D_k^{[j'],(m)} $ and Proposition \ref{prop:cmOj}, we have 
\begin{equation}
 c_{m} ( -D_k -  \sum _{j'=1}^{\ell  } D_k^{[j'],(m+1)} s^{j'} ) = c_{m} ( -D_k -  \sum _{j'=1}^{\ell  -1} D_k^{[j'],(m)} s^{j'} - D_k^{[\ell ],(m+1)} s^{\ell }  ) = O(s^{\ell} ) .
\end{equation}
Hence, we have $s F_{m+1} c_{m} ( -D_k -  \sum _{j'=1}^{\ell  } D_k^{[j'],(m+1)} s^{j'} ) = O(s^{\ell +1} )$.
By the definition of $D_k^{[j'],(m+1)} $, we have $ c_{m+1} (-D_k -  \sum _{j'=1}^{\ell  } D_k^{[j'],(m+1)} s^{j'} ) = O(s^{\ell +1} ) $.
We substitute $B= -D_k - \sum _{j'=1}^{\ell  } D_k^{[j'],(m+1)} s^{j'}  $ in equation (\ref{eq:cm+210}).
Then, we have $ c_{m+2} (-D_k -  \sum _{j'=1}^{\ell  } D_k^{[j'],(m+1)} s^{j'} ) = O(s^{\ell +1} ) $.
It follows from the definition of $D_k^{[j'],(m+2)} $ and Proposition \ref{prop:cmOj} that $D_k^{[j'],(m+1)}  = D_k^{[j'],(m+2)}  $ for $j' \leq \ell $.
\end{proof}
It follows from Proposition \ref{prop:Dkjm+1} that the expansion of the zero of $c_{m+1} (B) $ labelled by $-D_k $ up to the term $s^j$ does not depend on $m$, if $j \leq m-k $.
Namely, the expansion is written as
\begin{align}
& -D_k - D_k^{[1],(k+2)} s - D_k^{[2],(k+3)} s^2 - \cdots - D_k^{[j],(k+j+1)} s^{j}  +O(s^{j+1})  
\end{align}
for $j \leq m-k$.
Recall that the function $d_2(B)$ was defined by equation (\ref{eq:d1Bd2B}) and it is related to $c_{m+1} (B) $ by equation (\ref{eq:limkckd2B}).
We may expect that some of the zeros $d_2(B)$ are very close to some of the zeros of $c_{m+1} (B) $, if $m$ is large, and we propose the following conjecture.
\begin{conj}
On the Heun equation, the confluent Heun equation and the reduced confluent Heun equation, the zeros of $d_2(B)$ are expanded as 
\begin{align}
&  -D_k - \sum _{j=1}^{\infty } D_k^{[j],(k+j+1)} s^{j} 
\end{align}
for $k=0,1,2,\dots $.
\end{conj}

We now discuss the explicit expression of $D_k^{[2],(m+1)}$.
\begin{prop}
Set
\begin{align}
& B_{m-2}^{(m+1)} = -D_{m-2} \label{eq:Bm+1m-2} \\
& \quad + \Bigl[ -E_{m-2}-\frac{(m-2)(m-3+\gamma)}{D_{m-2}-D_{m-3}}F_{m-2}-\frac{(m-1)(m-2+\gamma)}{D_{m-2}-D_{m-1}}F_{m-1}\Bigr] s \nonumber \\
& \quad + \Bigl[ \Bigl\{ \frac{(m-2)(m-3+\gamma) F_{m-2} }{(D_{m-2}-D_{m-3})^2} + \frac{(m-1)(m-2+\gamma) F_{m-1}}{(D_{m-2}-D_{m-1})^2} \Bigr\} \nonumber \\
& \quad \qquad \cdot \Bigl\{ E_{m-2} + \frac{(m-2)(m-3+\gamma) F_{m-2} }{D_{m-2}-D_{m-3}} + \frac{(m-1)(m-2+\gamma) F_{m-1}}{D_{m-2}-D_{m-1}} \Bigr\} \nonumber \\
& \quad \quad - \frac{(m-2)(m-3+\gamma) F_{m-2} }{(D_{m-2}-D_{m-3})^2} \Bigl\{ E_{m-3}+\frac{(m-3)(m-4+\gamma) F_{m-3}}{D_{m-2}-D_{m-4}} \Bigr\} \nonumber \\
& \quad \quad - \frac{(m-1)(m-2+\gamma) F_{m-1}}{(D_{m-2}-D_{m-1})^2} \Bigl\{ E_{m-1}+\frac{m(m-1+\gamma) F_{m}}{D_{m-2}-D_{m}} \Bigr\} \Bigr] s^2 \nonumber 
\end{align}
for $m \geq 4$, and
\begin{align}
& B_{1}^{(4)} = -D_{1} + \Bigl[ -E_{1}-\frac{\gamma}{D_{1}-D_{0}}F_{1}-\frac{2(1+\gamma)}{D_{1}-D_{2}}F_{2}\Bigr] s \\
& \quad + \Bigl[ \Bigl\{ \frac{\gamma F_{1} }{(D_{1}-D_{0})^2} + \frac{2(1+\gamma) F_{2}}{(D_{1}-D_{2})^2} \Bigr\}  \Bigl\{ E_{1} + \frac{\gamma F_{1} }{D_{1}-D_{0}} + \frac{2(1+\gamma) F_{2}}{D_{1}-D_{2}} \Bigr\} \nonumber \\
& \quad \quad - \frac{\gamma E_{0} F_{1} }{(D_{1}-D_{0})^2}  - \frac{2(1+\gamma) F_{2}}{(D_{1}-D_{2})^2} \Bigl\{ E_{2}+\frac{3(2+\gamma) F_{3}}{D_{1}-D_{3}} \Bigr\} \Bigr] s^2 , \nonumber \\
& B_{0}^{(3)} = -D_{0} + \Bigl[ -E_{0}-\frac{\gamma}{D_{0}-D_{1}}F_{1}\Bigr] s \\
& \quad + \frac{\gamma F_{1}}{(D_{0}-D_{1})^2} \Bigl[ E_{0} - E_{1} + \frac{\gamma F_{1}}{D_{0}-D_{1}}  - \frac{2(1+\gamma) F_{2}}{D_{0}-D_{2}} \Bigr] s^2 . \nonumber 
\end{align}
Then, we have $c_{m+1} (B_{m-2}^{(m+1)}) = O(s^3) $ for $m \geq 2$.
\end{prop}
\begin{proof}
Assume that $m \geq 4$.
Write
\begin{align}
& \tilde{c}_k = c_{k} (-D_{m-2} - D_{m-2}^{[1],(m+1)} s +A s^2 ) , \; \tilde{F}_k = \frac{F_k }{(k+1)(k+\gamma )} , \\
& \tilde{E}_k = \frac{D_k -D_{m-2} +  (E_k - D_{m-2}^{[1],(m+1)}) s +A s^2 }{(k+1)(k+\gamma )}  . \nonumber
\end{align}
Then, it follows from the three-term recursion relation that
\begin{align}
& \tilde{c}_{k+1} = \tilde{E}_{k} \tilde{c}_{k} - s \tilde{F}_{k} \tilde{c}_{k-1} .
\end{align}
We have
\begin{align}
& \tilde{c}_{m+1} = \tilde{E}_m \tilde{c}_{m} - s \tilde{F}_m \tilde{c}_{m-1} 
 = ( \tilde{E}_m \tilde{E}_{m-1} - s \tilde{F}_m  ) \tilde{c}_{m-1} - s \tilde{E}_m \tilde{F}_{m-1} \tilde{c}_{m-2} \\
& = ( \tilde{E}_m \tilde{E}_{m-1} \tilde{E}_{m-2} -s  \tilde{F}_m   \tilde{E}_{m-2} -s  \tilde{E}_m \tilde{F}_{m-1} ) \tilde{c}_{m-2} + ( - s \tilde{E}_m \tilde{E}_{m-1} + s^2  \tilde{F}_m  ) \tilde{F}_{m-2} \tilde{c}_{m-3} , \nonumber \\
& \tilde{c}_{m-3} = \tilde{c}_{m-2} \frac{\tilde{c}_{m-3}}{\tilde{c}_{m-2}}  = \tilde{c}_{m-2} \frac{1}{\displaystyle \tilde{E}_{m-3}  -s  \tilde{F}_{m-3} \frac{1}{\tilde{E}_{m-4} -s  \tilde{F}_{m-4} \frac{\tilde{c}_{m-5}}{\tilde{c}_{m-4}} } } \nonumber \\
& \qquad = \tilde{c}_{m-2} \Bigl\{ \frac{1}{\tilde{E}_{m-3}} + \frac{\tilde{F}_{m-3}}{\tilde{E}_{m-3}^2\tilde{E}_{m-4}} s +O(s^2)  \Bigr\} . \nonumber 
\end{align}
Since $ \tilde{c}_{m-2} = O(s^0)$, the condition $\tilde{c}_{m+1} = O(s^3)  $ follows from
\begin{align}
& \tilde{E}_m \tilde{E}_{m-1} \tilde{E}_{m-2} -s  \tilde{F}_m   \tilde{E}_{m-2} -s  \tilde{E}_m \tilde{F}_{m-1} \\
&  + ( - s \tilde{E}_m \tilde{E}_{m-1} + s^2  \tilde{F}_m  ) \tilde{F}_{m-2} \Bigl( \frac{1}{\tilde{E}_{m-3}} + \frac{\tilde{F}_{m-3}}{\tilde{E}_{m-3}^2\tilde{E}_{m-4}} s  \Bigr) = O(s^3) . \nonumber
\end{align}
By dividing by $\tilde{E}_m \tilde{E}_{m-1} $, it is equivalent to
\begin{align}
& \frac{(E_{m-2} - D_{m-2}^{[1],(m+1)}) s +A s^2}{(m-1)(m-2+\gamma )} -s \frac{\tilde{F}_{m-1}}{ \tilde{E}_{m-1}} - s \frac{\tilde{F}_{m-2} }{\tilde{E}_{m-3}}  -s^2  \frac{\tilde{F}_m  }{ \tilde{E}_m \tilde{E}_{m-1}}  \frac{ (E_{m-2} - D_{m-2}^{[1],(m+1)}) }{(m-1)(m-2+\gamma )} \\
& - s^2 \frac{\tilde{F}_{m-2} \tilde{F}_{m-3}}{\tilde{E}_{m-3}^2\tilde{E}_{m-4}} + s^2 \frac{\tilde{F}_m \tilde{F}_{m-2} }{ \tilde{E}_m \tilde{E}_{m-1} \tilde{E}_{m-3}} = O(s^3) . \nonumber 
\end{align}
We obtain the value $A$ by a straightforward calculation, and we have equation (\ref{eq:Bm+1m-2}).
The cases $m=2,3$ are shown directly.
\end{proof}
By combining with Proposition \ref{prop:Dkjm+1}, we obtain the following theorem.
\begin{thm} \label{thm:expa2}
On the solution to $c_{m+1}(B)= 0 $ expressed as
\begin{align}
& B = -D_{k} - D_k^{[1],(m+1)} s - D_k^{[2],(m+1)} s^2 + O(s^2), 
\end{align}
we have
\begin{align}
& D_k^{[2],(m+1)} =-  \Bigl\{ \frac{k(k-1+\gamma) F_{k} }{(D_{k}-D_{k-1})^2} + \frac{(k+1)(k+\gamma) F_{k+1}}{(D_{k}-D_{k+1})^2} \Bigr\} \\
& \quad \qquad \cdot \Bigl\{ E_{k} + \frac{k(k-1+\gamma) F_{k} }{D_{k}-D_{k-1}} + \frac{(k+1)(k+\gamma) F_{k+1}}{D_{k}-D_{k+1}} \Bigr\} \nonumber \\
& \quad \quad + \frac{k(k-1+\gamma) F_{k} }{(D_{k}-D_{k-1})^2} \Bigl\{ E_{k-1}+\frac{(k-1)(k-2+\gamma) F_{k-1}}{D_{k}-D_{k-2}} \Bigr\} \nonumber \\
& \quad \quad + \frac{(k+1)(k+\gamma) F_{k+1}}{(D_{k}-D_{k+1})^2} \Bigl\{ E_{k+1}+\frac{(k+2)(k+1+\gamma) F_{k+2}}{D_{k}-D_{k+2}} \Bigr\}  \nonumber 
\end{align}
for $k=2,\dots ,m-2$, and
\begin{align}
& D_1^{[2],(m+1)}=  - \Bigl\{ \frac{\gamma F_{1} }{(D_{1}-D_{0})^2} + \frac{2(1+\gamma) F_{2}}{(D_{1}-D_{2})^2} \Bigr\} \Bigl\{ E_{1} + \frac{\gamma F_{1} }{D_{1}-D_{0}} + \frac{2(1+\gamma) F_{2}}{D_{1}-D_{2}} \Bigr\} \\
& \qquad \qquad + \frac{\gamma E_{0} F_{1} }{(D_{1}-D_{0})^2} + \frac{2(1+\gamma) F_{2}}{(D_{1}-D_{2})^2} \Bigl\{ E_{2}+\frac{3(2+\gamma) F_{3}}{D_{1}-D_{3}} \Bigr\} , \nonumber \\
& D_0^{[2],(m+1)} = \frac{\gamma F_{1}}{(D_{0}-D_{1})^2} \Bigl[ E_{1} - E_{0} - \frac{\gamma F_{1}}{D_{0}-D_{1}} + \frac{2(1+\gamma) F_{2}}{D_{0}-D_{2}} \Bigr]  . 
\end{align}
\end{thm}
Note that the expression of $D_k^{[1],(m+1)}$ was described in Theorem \ref{thm:B}.

\section{Lam\'e equation and numerical calculation} \label{sec:Lame}

An algebraic form of Lam\'e's differential equation is written as
\begin{align}
  \frac{d^{2}w}{dz^{2}} +\frac{1}{2}\left( \frac{1}{z } +\frac{1}{z -1} +\frac{1}{z -t}\right)\frac{dw}{dz} +\frac{\eta -n(n +1)z}{4z (z -1)(z -t)}w= 0.
\end{align}
It is a special case of Heun's differential equation in the case $\gamma = \delta = \epsilon = 1/2 $ and $\{\alpha, \beta\}= \{ (n+1)/2 , -n/2 \}$.
Then, the condition $\gamma +\delta +\epsilon =\alpha +\beta +1 $ is satisfied.

Set $\eta =-4 Bt$, $t=1/s$ and 
\begin{equation}
 y = \sum _{k=0}^{\infty } c_k (B)  z^k .
\end{equation}
Then, we have 
\begin{align}
& c_{m+1} (B) = \frac{1}{(m+1)(m+\gamma )} \{ (B+D_m +s E_m) c_{m} (B) -s F_m c_{m-1} (B) \} , 
\end{align}
where
\begin{align}
& D_{m}= m^2 ,\; E_{m}= m^2,  \; F_{m}= (2m+n-1)(2m-n-2)/4 , \; c_{-1} (B)=0. 
\end{align}
It follows from Theorem \ref{thm:expa2} that a large part of the zeros of $c_{m+1} (B) $ are expanded as
\begin{align}
& -k^2+ \Bigl\{ \frac{k^{2}}{2} -\frac{n(n+1)}{8} \Bigr\} s + \Bigl\{ \frac{3 k^{2}}{32} -\frac{n(n+1)}{64} - \frac{n^2(n+1)^2}{128(4 k^2 -1) }  \Bigr\} s^2 + O(s^3)
\label{eq:BkLame}
\end{align}
for $k= 2,\cdots, m-2  $, and the zeros corresponding to $k=0,1$ are written as
\begin{align}
&  \Bigl\{ -\frac{n(n+1)}{8} \Bigr\} s + \Bigl\{ -\frac{n(n+1)}{64} + \frac{n^2(n+1)^2}{128 } \Bigr\} s^2 + O(s^3) , \\
& -1 + \Bigl\{ \frac{1}{2} -\frac{n(n+1)}{8} \Bigr\} s + \Bigl\{ \frac{3 }{32} -\frac{n(n+1)}{128}  - \frac{5 n^2(n+1)^2}{768 }  \Bigr\} s^2 + O(s^3) .
\end{align}

We try to investigate the zeros of $c_{m+1} (B)$ numerically by setting the parameters to the particular values, and compare with the results by the approximation from the expansion in equation (\ref{eq:BkLame}).
We set
\begin{align}
& n=2.
\end{align}
A large part of the zeros of $c_{m+1} (B) $ are expanded as
\begin{align}
& -k^2+ \Bigl( \frac{k^{2}}{2} -\frac{3}{4} \Bigr) s + \Bigl( \frac{3 k^{2}}{32} -\frac{3}{32} - \frac{9}{32(4 k^2 -1) } \Bigr) s^2 + O(s^3)
\end{align}
for $k= 2,\cdots, m-2  $.
Hence, the zeroth order approximation is $-k^2$, the first order approximation is $-k^2+ (k^{2}/2 -3/4 ) s $, and the second order approximation is $-k^2+ (k^{2}/2 -3/4 ) s + \{ 3 k^{2}/32 - 3/32 - 9/(32(4 k^2 -1)) \} s^2 $.
The zeros corresponding to $k=0,1$ are written as
\begin{align}
& -\frac{3}{4}  s +  \frac{3}{16 }s^2 + O(s^3) , \;  -1  -\frac{1}{4} s - \frac{3 }{16} s^2 + O(s^3) .
\end{align}

On the other hand, we calculate the zeros of $c_{m+1} (B)$ numerically by fixing the value $s$ to the particular value.
We discuss the case $s=1/100$, which is equivalent to $t=100$.
The polynomial $c_4 (B)$ is written as
\begin{align}
 c_{4}(B) = \frac{2}{315} B^4 +\frac{101}{1125}B^3 +\frac{497299}{1575000}B^2 +\frac{6154031}{26250000}B +\frac{121537}{70000000},
\end{align}
and the zeros of $c_{4}(B) $ is calculated numerically as
\begin{align*}
& -.007481156136, \; -1.002518844, \; -3.988544101, \; -9.141455899 .
\end{align*}
Here, we used the maple and the command "fsolve".
The zeros of $c_{8} (B)$ are calculated numerically as
\begin{align*}
& -.007481156136, \; -1.002518844, \; -3.987473618, \; -8.962425455, \\
& -15.92736843 , \; -24.88377990, \; -35.92102641, \; -50.70792619.
\end{align*}
A part of the zeros of $c_{30} (B)$ are calculated numerically as
\begin{align*}
& -.007481156136, \; -1.002518844, \; -3.987473618, \; -8.962425430, \\
& -15.92735912 , \; -24.88227415, \; -35.82717042, \; \dots , \; -920.1619370 .
\end{align*}
The zeros of $c_{40} (B)$ are calculated similarly, and more than 20 zeros of $c_{40} (B)$ coincide with those of the zeros of $c_{30} (B)$ numerically up to 10 digits.
Therefore, we can expect that each zero of $d_2 (B)$ in equation (\ref{eq:d1Bd2B}) is a limit of a zero of $c_{m+1} (B)$ as $m \to \infty $.
We summarize the approximations of the zeros with respect to the expansions on the parameter $s(=1/100)$ and the numerical calculations of the zeros as the following table.
\begin{center}
\begin{tabular}{c|c|c|c|c}
& $0$th approx. & $1$st approx. & $2$nd approx. & zero of $ c_{30} (B) $ \\ \hline
$k=0$ & $0$ & $-.0075000000 $ & $-.0074812500 $ & $-.007481156136$ \\
$k=1$ & $-1$ & $-1.002500000 $ & $-1.002518750 $ & $-1.002518844$ \\
$k=2$ & $-4$ & $-3.987500000 $ & $-3.987473750 $ & $-3.987473618$ \\
$k=3$ & $-9$ & $-8.962500000 $ & $-8.962425804 $ & $-8.962425430$
\end{tabular}
\end{center}
Since $(1/100)^2=0.0001$, the second order approximation is reasonable.

We discuss the case $s=1/2$, which is equivalent to $t=2$.
The zeros of $c_{4}(B) $ are calculated numerically as
\begin{align*}
& -.3169872981, \; -1.183012702, \; -4.535836596, \; -14.96416340.
\end{align*}
The zeros of $c_{8} (B)$ are calculated numerically as
\begin{align*}
& -.3169872981, \; -1.183012702, \; -3.404179955, \; -7.875606843, \\
& -15.98660376 , \; -30.00270451, \; -53.91569102, \; -97.31521392.
\end{align*}
A part of the zeros of $c_{30} (B)$ are calculated numerically as
\begin{align*}
& -.3169872981, \; -1.183012702, \; -3.284830016, \; -6.870001746, \\
& -11.89319364 , \; -18.35299952, \; -26.25221475, \; \dots , \; -2042.087995 .
\end{align*}
A part of the zeros of $c_{40} (B)$ are calculated numerically as
\begin{align*}
& -.3169872981, \; -1.183012702, \; -3.284829947, \; -6.869999689, \\
& -11.89315665 , \; -18.35252588, \; -26.24770357, \; \dots , \; -3798.692942 .
\end{align*}
We may expect that each zero of $c_{m+1} (B)$ converges as $m \to \infty $, although the convergence speed is worse than the case $s=1/100$.
We summarize the approximations with respect to the expansions on the parameter $s(=1/2)$ and the numerical calculation as the following table.
\begin{center}
\begin{tabular}{c|c|c|c|c}
& $0$th approx. & $1$st approx. & $2$nd approx. & zero of $ c_{40} (B) $ \\  \hline
$k=0$ & $0$ & $-.3750000000 $ & $-.3281250000 $ & $-.3169872981$ \\
$k=1$ & $-1$ & $-1.125000000 $ & $-1.171875000 $ & $-1.183012702$ \\
$k=2$ & $-4$ & $-3.375000000 $ & $-3.309375000 $ & $-3.284829947$ \\
$k=3$ & $-9$ & $-7.125000000 $ & $-6.939508929 $ & $-6.869999689$
\end{tabular}
\end{center}
The second order approximation is not good for the case $s=1/2$.

\section{Mathieu equation, Whittaker-Hill equation and numerical calculation} \label{sec:MWH}

\subsection{Mathieu equation} 
The Mathieu equation is written as
\begin{equation}
\frac{d^2y}{dx^2} + ( a -2q \cos (2x)) y=0.
\end{equation}
By setting $z = \sin ^2 x$, we obtain the algebraic form of Mathieu's equation
\begin{equation}
z(1-z)\frac{d^2y}{dz^2} + \frac{1}{2}(1-2z) \frac{dy}{dz} + \frac{1}{4}(a-2q(1-2z)) y=0 ,
\end{equation}
which is a special case of the reduced singly confluent Heun equation
\begin{equation}
\frac{d^2y}{dz^2} + \left( \frac{\gamma}{z}+\frac{\delta }{z-1} \right) \frac{dy}{dz} +\frac{-s z + B }{z(z-1)} y=0 .
\end{equation}
with the condition
\begin{equation}
\gamma = \delta =1/2, \; s=q , \; B = q/2-a/4 .
\end{equation}
We describe a result in Section \ref{sec:pert} on the reduced singly confluent Heun equation.
The recursive relation for the polynomial $c_{m+1} (B)$ was written in equation (\ref{eq:cm+1rCH}).
It follows from Theorem \ref{thm:expa2} that a large part of the zeros of $c_{m+1} (B) $ are expanded as
\begin{align}
&  -k(k-1+\gamma +\delta ) + \Bigl\{ \frac{1}{2} +\frac{(\gamma -\delta )(\gamma +\delta -2 ) }{2 (2 k-2+\gamma +\delta ) (2 k+\gamma +\delta ) } \Bigr\} s \label{eq:BkRCHE} \\
& \quad + \Bigl\{ - \frac{1}{8} +\frac{3}{4} \frac{ (\gamma -1 )^2 +(\delta  -1 )^2 }{ (2 k-2+\gamma +\delta ) (2 k+\gamma +\delta )} -\frac{5}{8} \frac{(\gamma -\delta )^2 (\gamma +\delta -2 ) ^2}{ (2 k-2+\gamma +\delta )^2 (2 k+\gamma +\delta )^2} \nonumber \\
& \qquad -\frac{3}{2} \frac{(\gamma -\delta )^2 (\gamma +\delta -2 )^2}{ (2 k-2+\gamma +\delta )^3 (2 k+\gamma +\delta )^3}  \Bigr\} \frac{s^2 }{(2 k-3+\gamma +\delta ) (2 k+1 +\gamma +\delta )  } + O(s^3)  \nonumber 
\end{align}
for $k= 2,\cdots, m-2  $.

We try to investigate the zeros of $c_{m+1} (B)$ numerically by setting the parameters to the particular values, and compare with the results by the approximation from the expansion in equation (\ref{eq:BkRCHE}).
We restrict to the case
\begin{align}
& \gamma =\delta = 1/2 .
\end{align}
In this case, a large part of the zeros of $c_{m+1} (B) $ are written as
\begin{align}
&  -k^2+ \frac{1}{2} s - \frac{1}{8(4 k^2 -1) } s^2 + O(s^3) .
\end{align}
for $k= 2,\cdots, m-2  $.
The zeros corresponding to $k=0,1$ are written as
\begin{align}
& \frac{1}{2} s + \frac{1}{8 } s^2 + O(s^3) , \;  -1 +\frac{1}{2} s - \frac{5 }{48} s^2 + O(s^3) .
\end{align}
We calculate the zeros of $c_{m+1} (B)$ numerically by fixing the value $s$ to the particular value.
We discuss the case $s=2$.
The zeros of $c_{8} (B)$ are calculated numerically as
\begin{align*}
& 1.378487800, \; -.2931696155, \; -3.035348462, \; -8.015086274, \\
& -15.02054999 , \; -24.16203960, \; -36.25913692, \; -54.59315694.
\end{align*}
A part of the zeros of $c_{30} (B)$ are calculated numerically as
\begin{align*}
& 1.378489221, \; -.2931662833, \; -3.035300946, \; -8.014303906, \\
& -15.00793924 , \; -24.00505119, \; -35.00349673, \; \dots , \; -864.9717520 .
\end{align*}
The zeros of $c_{40} (B)$ are calculated similarly, and more than 20 zeros of $c_{40} (B)$ coincide with those of the zeros of $c_{30} (B)$ numerically up to 10 digits.
Therefore, we can expect that each zero of $d_2 (B)$ in equation (\ref{eq:d1Bd2B}) is a limit of a zero of $c_{m+1} (B)$ as $m \to \infty $.
We summarize the approximations of the zeros with respect to the expansions on the parameter $s(=2)$ and the numerical calculations of the zeros as the following table.
\begin{center}
\begin{tabular}{c|c|c|c|c}
& $0$th approx. & $1$st approx. & $2$nd approx. & zero of $ c_{30} (B) $ \\ \hline
$k=0$ & $0$ & $1 $ & $1.500000000 $ & $1.378489221$ \\
$k=1$ & $-1$ & $0 $ & $-.4166666667 $ & $-.2931662833$ \\
$k=2$ & $-4$ & $-3 $ & $-3.033333333 $ & $-3.035300946$ \\
$k=3$ & $-9$ & $-8 $ & $-8.014285714 $ & $-8.014303906$ \\
$k=4$ & $-16$ & $-15 $ & $-15.00793651 $ & $-15.00793924$ \\
$k=5$ & $-25$ & $-24 $ & $-24.00505051 $ & $-24.00505119$
\end{tabular}
\end{center}
The $2$nd order approximation is not good for small $k$, but it seems that the approximation would be better if $k$ is getting large.

Let $i = \sqrt{-1}$ and we discuss the case $s=2i$.
The zeros of $c_{8} (B)$ are calculated numerically as
\begin{align*}
& -.5406371066+.5331266835 i , \; -.5406402496+1.466879047 i, \\
& -3.968636255+.9999756999 i, \; -8.986123112+.9989291467 i, \\
& -16.01021544+1.007508356 i , \; -24.84797965+1.233162247 i , \\
& -34.56459475-.1600728404 i , \; -50.54117344-6.079508341 i .
\end{align*}
A part of the zeros of $c_{30} (B)$ are calculated numerically as
\begin{align*}
& -.5406395812+.5331266960 i , \; -.5406395812+1.466873304 i , \; -3.968701175+ i, \\ 
& -8.985730155+i , \; -15.99206621+i , \; -24.99495018+i , \; -35.99650372+i , \dots , \\
& -846.6304900-26.02805043 i .
\end{align*}
More than 20 zeros of $c_{40} (B)$ coincide with those of the zeros of $c_{30} (B)$ numerically up to 10 digits.

We can also expect that each zero of $c_{m+1} (B)$ converges as $m \to \infty $ for the case $s=2i$.
We summarize the approximations of the zeros with respect to the expansions on the parameter $s(=2i)$ and the numerical calculations of the zeros as the following table.
\begin{center}
\begin{tabular}{c|c|c|c|c}
& $0$th approx. & $1$st approx. & $2$nd approx. & zero of $ c_{30} (B) $ \\ \hline
$k=0$ & $0$ & $i $ & $-.5000000000+i $ & $ -.5406395812+.5331266960 i$ \\
$k=1$ & $-1$ & $-1 + i $ & $-.5833333333+i $ & $-.5406395812+1.466873304 i$ \\
$k=2$ & $-4$ & $-4 + i $ & $-3.966666667+i $ & $-3.968701175+ i$ \\
$k=3$ & $-9$ & $-9 + i $ & $-8.985714286+i $ & $-8.985730155+i$ \\
$k=4$ & $-16$ & $-16 + i $ & $-15.99206349+i $ & $-15.99206621+i$ \\
$k=5$ & $-25$ & $-25 + i $ & $-24.99494949+i $ & $-24.99495018+i$
\end{tabular}
\end{center}
Therefore, the approximation is not good for $k=0$ and $1$.

The real parts of the first two numerical zeros for $c_{30} (B)$ coincide up to 10 digits, although the imaginary parts are $i$ for $k=0$ and $k=1$ on the $2$nd order approximation. 

We discuss the case $s=i$.

The zeros of $c_{8} (B)$ are calculated numerically as
\begin{align*}
& -.1431861828+.4999999951 i , \; -.8775200607+.5000000522 i, \\
& -3.991791404+.4999998065 i, \; -8.996435830+.4999651877 i, \\
& -15.99919580+.5002363252 i , \; -24.99010545+.5324971426 i , \\
& -35.50299960+.3000806539 i , \; -49.49876567-3.332779163 i .
\end{align*}
A part of the zeros of $c_{30} (B)$ are calculated numerically as
\begin{align*}
& -.1431861712+.5 i , \; -.8775200792+.5 i , \; -3.991792466+.5 i , \; -8.996429618+.5 i  , \\
& -15.99801604+.5 i , \; -24.99873742+.5 i , \; -35.99912589+.5 i , \; \dots , \\
& -842.7448796-13.99121594 i .
\end{align*}
More than 20 zeros of $c_{40} (B)$ coincide with those of the zeros of $c_{30} (B)$ numerically up to 10 digits.
We can also expect that each zero of $c_{m+1} (B)$ converges as $m \to \infty $ for the case $s=i$.
We summarize the approximations of the zeros with respect to the expansions on the parameter $s(=i)$ and the numerical calculations of the zeros as the following table.
\begin{center}
\begin{tabular}{c|c|c|c|c}
& $0$th approx. & $1$st approx. & $2$nd approx. & zero of $ c_{30} (B) $ \\ \hline
$k=0$ & $0$ & $.5 i $ & $-.1250000000+.5i $ & $  -.1431861712+.5 i $ \\
$k=1$ & $-1$ & $-1 + .5 i $ & $-.8958333333+.5 i $ & $-.8775200792+.5 i $ \\
$k=2$ & $-4$ & $-4 + .5 i $ & $-3.991666667+.5 i $ & $-3.991792466+.5 i $ \\
$k=3$ & $-9$ & $-9 + .5 i $ & $-8.996428571+.5 i $ & $-8.996429618+.5 i $ \\
$k=4$ & $-16$ & $-16 + .5 i $ & $-15.99801587+.5 i $ & $-15.99801604+.5 i $ \\
$k=5$ & $-25$ & $-25 + .5 i $ & $-24.99873737+.5 i $ & $-24.99873742+.5 i $
\end{tabular}
\end{center}
By observing the zeros for two cases $s=2i$ and $s=i$, bifurcation structure is expected around $s=a i$ where $1<a<2$.
In fact, if $s = i 1.468768613785 \cdots $, then two eigenvalues of the Mathieu equation coalesce and the value $s = i 1.468768613785 \cdots $ is called Mulholland-Goldstein double point.
See \cite{BCZ} and references therein.

\subsection{Whittaker-Hill equation}
The Whittaker-Hill equation is written as
\begin{equation}
\frac{d^2y}{dx^2} + ( A_0 + A_1 \cos (2x) + A_2 \cos (4x) ) y=0 .
\label{eq:WhHi}
\end{equation}
See \cite{Ars}.
By setting $z = \sin ^2 x$, we obtain 
\begin{equation}
4z(1-z)\frac{d^2 w}{dz^2} + 2(1-2z) \frac{dw}{dz} + \{ A_0 + A_1 (1-2z) + A_2 (8 z^2 -8z +1 )  \} w =0 .
\end{equation}
To fit it with the singly confluent Heun equation, we apply a gauge transformation.
Set $w= y \exp (h z) $ and $A_2 =h^2/2$.
Then, equation (\ref{eq:WhHi}) is equivalent to
\begin{equation}
\frac{d^2y}{dz^2} + \left( 2h + \frac{1/2}{z}+\frac{1/2 }{z-1} \right) \frac{dy}{dz} +\frac{ 4( A_1 + 2h )z  -(2A_0 + 2A_1 + 4 h + h^2 )}{8 z(z-1)} y=0 ,
\end{equation}
and it is a special case of the singly confluent Heun equation
\begin{equation}
\frac{d^2y}{dz^2} + \left( -s + \frac{\gamma}{z}+\frac{\delta }{z-1} \right) \frac{dy}{dz} +\frac{-s \alpha z + B }{z(z-1)} y=0,
\end{equation}
with the condition
\begin{equation}
\gamma = \delta =1/2, \; s=-2h , \; \alpha = 1/2  +  A_1 /(4h) , \; B = -(2A_0 + 2A_1 + 4 h + h^2 )/8 .
\end{equation}
We describe a result in Section \ref{sec:pert} on the singly confluent Heun equation.
The recursive relation for the polynomial $c_{m+1} (B)$ was obtained in equations (\ref{eq:cm+1CH}) and (\ref{eq:CHDEF}).
It follows from Theorem \ref{thm:expa2} that the expansion of the zero of $c_{m+1} (B) $ is written as
\begin{align}
& -k(k-1+\gamma +\delta ) + \Bigl\{ \frac{\gamma -\delta +2 \alpha }{4} -\frac{(\gamma -\delta )(\gamma +\delta -2 )(\gamma +\delta -2 \alpha  ) }{4 (2 k-2+\gamma +\delta ) (2 k+\gamma +\delta )} \Bigr\} s \\
& \quad + \Bigl\{ - \frac{1}{32} + \frac{ P_6(k) }{(2 k-3+\gamma +\delta ) (2 k-2+\gamma +\delta )^3 (2 k+\gamma +\delta )^3  (2 k+1 +\gamma +\delta )  }  \Bigr\} s^2 \nonumber \\
& \quad + O(s^3) , \nonumber 
\end{align}
for $k=2,\dots , m-2$, where $P_6(k)$ is a polynomial in $k$ of degree $6$.

We try to investigate the zeros of $ c_{m+1} (B)$ numerically by setting the parameters to the particular values, and compare with the results by the approximation from the expansion in $s$.
We restrict to the case
\begin{align}
& \gamma =\delta = 1/2 .
\end{align}
Then, the expansion of the zero of $c_{m+1} (B) $  is written as
\begin{align}
& -k^2+ \frac{\alpha }{2} s + \Bigl\{ -\frac{1}{32} - \frac{(2\alpha -1)^2}{32(4 k^2 -1) }  \Bigr\} s^2 + O(s^3) 
\end{align}
for $k= 2,\cdots, m-2  $.
The zeros corresponding to $k=0,1$ is written as
\begin{align}
& \frac{\alpha }{2} s + \frac{\alpha ^2 - \alpha }{8 }  s^2  + O(s^3) , \;  -1 +\frac{\alpha }{2} s +  \frac{- 5 \alpha ^2 +5 \alpha  -2}{48} s^2 + O(s^3) .
\end{align}
We calculate the zeros of $c_{m+1} (B)$ numerically by fixing the values $s$ and $\alpha $.

We discuss the case $\alpha =5$ and $s=-1/100$.
The zeros of $c_{8} (B)$ is calculated numerically as
\begin{align*}
& -.02475005544, \; -1.025212444, \; -4.025020000, \; -9.025010357, \\
& -16.02500729 , \; -25.02496069, \; -36.03300188, \; -48.53703728.
\end{align*}
A part of the zeros of $c_{30} (B)$ is calculated numerically as
\begin{align*}
& -.02475005544, \; -1.025212444, \; -4.025020000, \; -9.025010357, \\
& -16.02500714 , \; -25.02500568, \; -36.02500490, \; \dots , \; -835.6811120 .
\end{align*}
More than 20 zeros of $c_{40} (B)$ coincides with those of the zeros of $c_{30} (B)$ numerically up to 10 digits.
Therefore, we can expect that each zero of $d_2 (B)$ in equation (\ref{eq:d1Bd2B}) is a limit of a zero of $c_{m+1} (B)$ as $m \to \infty $.

We summarize the approximations of the zeros with respect to the expansions on the parameter $s(=-1/100)$ and the numerical calculations of the zeros as the following table.
\begin{center}
\begin{tabular}{c|c|c|c|c}
& $0$th approx. & $1$st approx. & $2$nd approx. & zero of $ c_{30} (B) $ \\ \hline
$k=0$ & $0$ & $-.0250000000$ & $-.0247500000 $ & $-.02475005544$ \\
$k=1$ & $-1$ & $-1.025000000$ & $-1.025212500 $ & $-1.025212444$ \\
$k=2$ & $-4$ & $-4.025000000$ & $-4.025020000 $ & $-4.025020000$ \\
$k=3$ & $-9$ & $-9.025000000$ & $-9.025010357 $ & $-9.025010357$ \\
$k=4$ & $-16$ & $-16.02500000$ & $-16.02500714 $ & $ -16.02500714$ \\
$k=5$ & $-25$ & $-25.02500000$ & $-25.02500568 $ & $-25.02500568$
\end{tabular}
\end{center}
Since $(-1/100)^2=0.0001$, the second order approximation is reasonable.

We discuss the case $\alpha =5$ and $s=-20$.
A part of the zeros of $c_{16} (B)$ is calculated numerically as
\begin{align*}
& 2.051282315, \; 31.28075111, \; 54.63259197-14.18872668 i, \; \dots ,
\end{align*}
and other zeros are not real numbers.
A part of the zeros of $c_{19} (B)$ is calculated numerically as
\begin{align*}
& 2.051136527, \; 31.66623015, \; 50.75984059, \; 59.89047866-28.89267645i  , \; \dots ,
\end{align*}
and other zeros are not real numbers.
We might expect that $c_{m+1} (B)$ has a zero around $2.051$ as $m \to \infty$, but it would not be true.
It is shown by the numerical calculation that all zeros of $c_{50} (B)$ are not real numbers.
A part of the zeros of $c_{89} (B)$ is calculated numerically as
\begin{align*}
&-11.72870190, \; -37.26280325, \; -59.78916204, \; -78.71230515 ,\\
& -92.50702734, \; -97.17838365, \; -105.1776795 , \; -116.4825515, \dots .
\end{align*}
17 zeros of $c_{89} (B)$ are real numbers.
A part of the zeros of $c_{100} (B)$ is calculated numerically as
\begin{align*}
&-11.72870190, \; -37.26280325, \; -59.78916204, \; -78.71230514 ,\\
&-92.50702766, \; -97.17838248, \; -105.1776823 , \; -116.4825441, \dots .
\end{align*}
26 zeros of $c_{100} (B)$ are real numbers.
We expect that the polynomial $c_{m+1} (B)$ has zeros around $-11.72870190, \; -37.26280325, \; -59.78916204, \dots $, if $m$ is sufficiently large.

\section{Concluding remarks}  \label{sec:rmk}

In \cite{Tak2}, a perturbative approach to the Calogero-Moser-Sutherland system, which is related with Heun's differential equation, was performed, and it was applied to the Lam\'e equation in \cite{TakL}.
It would be interesting to compare the results in \cite{Tak2,TakL} to the ones in this paper.
It is known that the Lam\'e equation has the finite-gap property (see \cite{TakS} and references therein).
The numerical calculations in Section \ref{sec:Lame} would produce a part of the spectral of the Lam\'e equation as an approximation.
Details on this aspect should be revealed in a near future.
Some examples of the numerical calculations for the Mathieu equation and for the Whittaker-Hill equation were discussed in Section \ref{sec:MWH}, and the results would produce a part of the spectral of those equations as an approximation.
It would also be possible to carry out numerical calculations for the Heun equation, the singly confluent Heun equation and the reduced singly confluent Heun equation with the condition $\gamma \neq 1/2$ or $\delta \neq 1/2$, and we expect further studies for the spectral of them.

Although we obtained exact results on the zeros of the polynomial $c_{m+1} (B)$ from the perturbative approach in Section \ref{sec:pert}, the zeros were discussed numerically in Sections \ref{sec:Lame} and \ref{sec:MWH}, and we did not have exact estimates for error bounds.
In other words, we discussed that some zeros of $c_{m+1} (B)$ are sufficiently close to some zeros of $c_{m'+1} (B)$ numerically for some values $m$ and $m'$ by examples, but we did not discuss exact estimates of the zeros of $d_2(B)$ and limits of some zeros of $c_{m+1} (B)$ as $m \to \infty $.
Note that the zeros of $d_2(B)$ are directly related to the global structure of the solutions to the corresponding Heun-type equations.
To obtain exact estimates, it would be necessary to develop the analytic argument of Sch{\"a}fke and Schmidt \cite{SS} to be fit with the uniform convergence on compact sets with respect to the parameter $B$ for Heun-class equations.
The approach by Volkmer \cite{Vol} might be useful to discuss convergence of the zeros.

\section*{Acknowledgements}
The second author was supported by JSPS KAKENHI Grant Number JP22K03368.

\end{document}